\documentclass[11pt]{amsart}

\usepackage{amsmath, amsthm, amssymb, amsfonts, amscd, graphicx}

\def\CC{{\mathbb C}}
\def\FF{{\mathbb F}}

\def\QQ{{\mathbb Q}}

\def\QQ{{\mathbb Q}}
\def\RR{{\mathbb R}}

\def\ZZ{{\mathbb Z}}

\def\0{{\mathbf 0}}
\def\1{{\mathbf 1}}

\def\v{{\mathbf v}}

\def\x{{\mathbf x}}

\def\Fix{\mathrm{Fix}}

\def\ch{\mathrm{char}}

\def\Perm{\mathrm{Perm}}

\def\sgn{\mathrm{sgn}}

\def\ord{\mathrm{ord}}

\def\tr{\mathrm{tr}}

\def\max{\mathrm{max}}

\theoremstyle{plain}

\newtheorem{thm}{Theorem}

\newtheorem{lem}[thm]{Lemma}

\theoremstyle{definition}

\newtheorem{rem}{Remark}
\newtheorem{ex}{Example}

\title[A $p$-adic Perron-Frobenius Theorem]{A $p$-adic Perron-Frobenius Theorem}

\author{Robert Costa}

\address{Robert Costa; Tufts University; Department of Mathematics; 503 Boston Avenue; Medford, MA 02155}

\email{Robert\_J.Costa@tufts.edu}

\author{Patrick Dynes}

\address{Patrick Dynes; Department of Mathematical Sciences; O-110 Martin Hall; Clemson, S.C. 29634}

\email{pdynes@g.clemson.edu}

\author{Clayton Petsche}

\address{Clayton Petsche; Department of Mathematics; Oregon State University; Corvallis OR 97331 U.S.A.}

\email{petschec@math.oregonstate.edu}

\date{August 28 2015.  Revised February 11 2016}

\begin{document}

\begin{abstract}
We prove that if an $n\times n$ matrix defined over $\QQ_p$ (or more generally an arbitrary complete, discretely-valued, non-Archimedean field) satisfies a certain congruence property, then it has a strictly maximal eigenvalue in $\QQ_p$, and that iteration of the (normalized) matrix converges to a projection operator onto the corresponding eigenspace.  This result may be viewed as a $p$-adic analogue of the Perron-Frobenius theorem for positive real matrices.
\end{abstract}

\maketitle

\section{Introduction}

The Perron-Frobenius theorem gives information about the eigenvalues and eigenvectors of certain matrices with nonnegative real entries.  In 1907 Perron proved the simplest version of the result, for positive matrices, and in 1912 Frobenius extended the result to (irreducible) nonnegative matrices.  (Recall that a matrix is said to be {\em positive} (resp. {\em nonnegative}) if all of its entries are positive (resp. {\em nonnegative}) real numbers.)  Perron's version of the theorem can be stated as follows.  

\begin{thm}[\cite{MR2356919}, Ch.~16, Thm. 1]
\label{thm:perron_frobenius}
Let $A$ be a positive $n\times n$ matrix.   Then $A$ has a positive real eigenvalue $\lambda_\max$, of multiplicity one, such that $|\lambda| < \lambda_\max$ for all other complex eigenvalues $\lambda$ of $A$. Furthermore, there exists a positive $\lambda_\max$-eigenvector of $A$.
\end{thm}

A useful application of the Perron-Frobenius theorem is the following well-known result.  Following \cite{MR2356919}, an $n\times n$ matrix is said to be {\em stochastic} if it is nonnegative and the sum along each column of the matrix is equal to $1$.

\begin{thm}[\cite{MR2356919}, Ch.~16, Thm. 3]
\label{thm:stochastic_matrices}
Let $A$ be a positive $n\times n$ stochastic matrix. Then the maximal eigenvalue $\lambda_\max$ of $A$ is equal to $1$, and for any nonnegative vector $\x$, the sequence $\{A^k \x\}$ converges to a positive $\lambda_\max$-eigenvector of $A$.
\end{thm}

The Perron-Frobenius theorem and its application to the iteration of stochastic matrices have implications in many areas of mathematics, including graph theory, probability theory, and symbolic dynamical systems.  Real-world applications of these results include the modeling of changes in atomic nuclei and populations in ecological systems, and the PageRank algorithm used by Google as a basis for its internet search strategy.  See \cite{MR2356919} and \cite{MR1786935} for further background on the Perron-Frobenius theorem and its applications.

In this paper we give an analogue of (Perron's form of) the Perron-Frobenius theorem for $n\times n$ matrices defined over the field $\QQ_p$ of $p$-adic numbers.  Denote by $|\cdot|_p$ the $p$-adic absolute value on $\QQ_p$, normalized as usual so that $|p|_p=1/p$, and for each $c\in\QQ_p$ and real $r>0$, denote by $D(c,r)=\{x\in\QQ_p\mid|x-c|_p\leq r\}$ the closed $p$-adic disc with center $c$ and radius $r$.  As usual, denote by $\CC_p$ the completion of the algebraic closure of $\QQ_p$.

To formulate a $p$-adic analogue of Perron's theorem, we take the $p$-adic disc $D(1,1/p)$ as our analogue of the positive real numbers.  The motivation for this choice is simply the observation that $(0,+\infty)$ is the largest subinterval of $\RR$ containing $1$ but not containing $0$, and $D(1,1/p)$ is the largest disc in $\QQ_p$ containing $1$ but not containing $0$.  This leads to a fairly close parallel with Perron's theorem, and as we explain in Remark~\ref{GeneralizationRemark}, the choice of the disc $D(1,1/p)$ is not as limiting as it may at first appear.  The simplest nontrivial case of our main result is the following.

\begin{thm}\label{MainThmIntro}
Let $p$ be a prime number and let $n$ be a positive integer such that $p\nmid n$.  Let $A$ be an $n\times n$ matrix over $\QQ_p$ with all entries in the disc $D(1,1/p)$.  Then the following conditions hold.
\begin{itemize}
	\item[{\bf (a)}]  $A$ has a simple eigenvalue $\lambda_\max\in \QQ_p$ such that $|\lambda_\max|_p=1$, and $|\lambda|_p<1$ for all other eigenvalues $\lambda\in\CC_p$ of $A$.  
	\item[{\bf (b)}]  The maximal eigenvalue $\lambda_\max$ is an element of the disc $D(n,1/p)$.
	\item[{\bf (c)}]  There exists a $\lambda_\max$-eigenvector $\x$ of $A$, all of whose components are elements of the disc $D(1,1/p)$.
	\item[{\bf (d)}]  The ($p$-adic) limit
\begin{equation*}
P_A=\lim_{k\to+\infty}\textstyle(\frac{1}{\lambda_\max}A)^k
\end{equation*} 
exists and $P_A$ is a projection operator onto the $\lambda_\max$-eigenspace of $A$.
\end{itemize}
\end{thm}

\begin{rem}\label{GeneralizationRemark}
Theorem~\ref{MainThmIntro}, and its generalization Theorem~\ref{MainThmGeneral}, can be used to study a broader class of matrices than just those whose entries are near $1$.  For example, suppose that $B=(b_{ij})$ is an $n\times n$ matrix over $\QQ_p$, all of whose entries lie in a $p$-adic disc $D(\alpha,r)$ such that $0\not\in D(\alpha,r)$.  Then $|\alpha|_p>r$ and $|b_{ij}/\alpha-1|_p\leq r/|\alpha|_p<1$ for all $1\leq i,j\leq n$.  Therefore the matrix $A=\alpha^{-1}B$ satisfies the hypotheses of Theorem~\ref{MainThmIntro}, and the conclusions pertaining to $A$ easily imply corresponding statements about the matrix $B$.
\end{rem}

\begin{ex}
The $2\times 2$ matrix
\begin{equation*}
A=\begin{pmatrix}
4 & -5  \\
1 & 10
\end{pmatrix}
\end{equation*}
over $\QQ_3$ satisfies the hypotheses of Theorem~\ref{MainThmIntro}.  The characteristic polynomial of $A$ is
\begin{equation*}
f(x) = x^2 - 14x + 45=(x-9)(x-5).
\end{equation*}
Since $|9|_3=1/9$ and $|5|_3=1$, we have $\lambda_\max=5$.  The $5$-eigenvector ${-5\choose 1}$ of $A$ has both entries in $D(1,1/3)$, and the $3$-adic limit
\begin{equation*}
\lim_{k\to+\infty}\begin{pmatrix}
4/5 & -1  \\
1/5 & 2
\end{pmatrix}^k=
\begin{pmatrix}
5/4 & 5/4  \\
-1/4 & -1/4
\end{pmatrix}
\end{equation*}
exists and is equal to a projection operator onto the $5$-eigenspace of $A$.  The evaluation of this limit, as shown, follows from the diagonalization argument given in the proof of Theorem~\ref{MainThmIntro}; see especially $(\ref{ProjectionCalc})$.
\end{ex}

\begin{ex}\label{ExNotRational}
Theorem~\ref{MainThmIntro} guarantees that $\lambda_\max\in\QQ_p$, but of course it may happen that $\lambda_\max\notin\QQ$, as the example 
\begin{equation*}
A=\begin{pmatrix}
8 & 1  \\
1 & 1
\end{pmatrix}
\end{equation*}
over $\QQ_7$ shows.  This matrix satisfies the hypotheses of Theorem~\ref{MainThmIntro}, and its characteristic polynomial is $f(x) = x^2 -9x+7$, which is irreducible over $\QQ$.
\end{ex}

\begin{ex}\label{ExMinEVNotQpRational}
While $\lambda_\max\in\QQ_p$, the nonmaximal eigenvalues are not necessarily in $\QQ_p$.  Consider 
\begin{equation*}
A=\begin{pmatrix}
6 & 1 & -4 \\
1 & -4 & 6 \\
-4 & 6 & 1
\end{pmatrix}
\end{equation*}
over $\QQ_5$, which satisfies the hypotheses of Theorem~\ref{MainThmIntro}.  Its characteristic polynomial is $f(x) = x^3 -3x^2-75x+225=(x-3)(x^2-75)$.  Thus $\lambda_\max=3$, but $x^2-75$ is irreducible over $\QQ_5$.  Indeed, $3$ is not a square in $\QQ_5$ because it is not a square in the residue field $\FF_5=\{0,1,2,3,4\}$.  Therefore $75=5^2\cdot3$ is not a square in $\QQ_5$.
\end{ex}

\begin{ex}\label{CounterExIntro}
This example illustrates that, if the requirement $p\nmid n$ is omitted, then the theorem is false as written.  Let $p=n=2$ and consider the $2\times 2$ matrix
\begin{equation*}
A=\begin{pmatrix}
5/3 & 1  \\
1 & 7/3
\end{pmatrix}
\end{equation*}
over $\QQ_2$.  All of the entries are elements of the disc $D(1,1/2)$, and the characteristic polynomial of $A$ is
\begin{equation*}
f(x) = x^2 - 4x + 26/9.
\end{equation*}
A Newton polygon argument shows that the two eigenvalues $\lambda_1,\lambda_2\in\CC_2$ of $A$ both have absolute value $|\lambda_1|_2=|\lambda_2|_2=2^{-1/2}$.  Thus $A$ has no strictly maximal eigenvalue in $\CC_2$, and no eigenvalues at all in $\QQ_2$.
\end{ex}

Despite Example~\ref{CounterExIntro}, a version of Theorem~\ref{MainThmIntro} can be recovered, even when $p\mid n$, provided we strengthen the hypothesis that all entries of $A$ are near $1$.  Precisely, if we require that all entries of $A$ lie in the smaller disc $D(1,1/p^{\ell})$ for a positive integer $\ell>2\,\ord_p(n)$, then (a slightly modified version of) Theorem~\ref{MainThmIntro} continues to hold.  We also prove, in $\S$~\ref{SharpExampleSect}, that this result is sharp by giving a counterexample if the hypothesis $\ell>2\,\ord_p(n)$ is weakened.  

Since it requires no extra work, we prove our results in the setting of an arbitrary complete, discretely-valued, non-Archimedean field.  Thus in $\S$~\ref{PrelimSect} we briefly review some notation and basic facts about such fields, and we prove two crucial lemmas.  In $\S$~\ref{MainResultSect} we prove our main result, Theorem~\ref{MainThmGeneral}, of which Theorem~\ref{MainThmIntro} is a special case.  In $\S$~\ref{SharpExampleSect} we give a counterexample to our main result if the main hypothesis is weakened.

This work was done during the Summer 2015 REU program in Mathematics at Oregon State University, with support by National Science Foundation Grant DMS-1359173.

\section{Notation and Preliminary Lemmas}\label{PrelimSect}

Throughout this paper, $K$ denotes a field which is complete with respect to a nontrivial, non-Archimedean absolute value $|\cdot|$, and which has a discrete value group $|K^\times|$.  This means that, in addition to the usual axioms (\cite{MR1488696} $\S$2.1) satisfied by an absolute value, $|\cdot|$ also satisfies the strong triangle inequality 
\begin{equation}\label{STI}
|x+y|\leq\max(|x|,|y|),
\end{equation}
and a standard argument (\cite{MR1488696} Prop. 2.3.3) shows that 
\begin{equation}\label{STIEquality}
|x+y|=\max(|x|,|y|) \hskip1cm (\text{whenever } |x|\neq|y|).
\end{equation}

Let $R=\{x\in K\mid |x|\leq 1\}$ denote the ring of integers in $K$ and let $\pi$ be a uniformizing paramater.  This means that $\pi$ is an element of $R$ satisfying
\begin{equation*}
|\pi|=\max\{|x|\,\mid x\in R,|x|<1\},
\end{equation*} 
or equivalently, that $\pi R$ is the unique maximal ideal of $R$.  (Such an element $\pi$ exists by the assumption that $K$ is discretely valued.)  Denote by $\ord_\pi$ the discrete valuation on $K$ normalized so that $|x|=|\pi|^{\ord_\pi(x)}$ for all $x\in K$; thus $\ord_\pi(\pi)=1$.

Given an element $c\in K$ and a real number $r\geq0$, denote by
\begin{equation*}
D(c,r) = \{x\in K\,\mid\,|x-r|\leq r\}
\end{equation*} 
the closed disc in $K$ with center $c$ and radius $r$.

Given a matrix $A=(a_{ij})$ with entries in $K$, define its norm by
\begin{equation}\label{NormDef}
\|A\|=\max\,|a_{ij}|
\end{equation}
This defines a norm on the vector space $M_{m\times n}(K)$ of $m\times n$ matrices over $K$.  The strong triangle inequality implies the bound 
\begin{equation}\label{NormBound}
\|A B\|\leq\|A\|\|B\|
\end{equation}
whenever $A\in M_{m\times n}(K)$ and $B\in M_{n\times r}(K)$.

Let $\CC_K$ be the completion of the algebraic closure of $K$.  The field $\CC_K$ is both complete and algebraically closed, and the absolute value $|\cdot|$ extends uniquely to a non-Archimedean absolute value on $\CC_K$ \cite{bosch:nonarchanalysis}.  We extend the norm $\|\cdot\|$ to matrices defined over $\CC_K$ via the formula $(\ref{NormDef})$.

\begin{lem}\label{MainDetLem}
Let $A$ be an $n\times n$ matrix over $K$ with all entries in the disc $D(1,|\pi|^\ell)$ for some positive integer $\ell$.  Then 
\begin{equation}\label{SmallDeterminant}
|\det A|\leq|\pi|^{\ell(n-1)}.
\end{equation}
\end{lem}
\begin{proof}
By hypothesis we may write
\begin{equation}\label{AWrittenNear1}
A=(1+\pi^\ell a_{ij}) \hskip1cm (a_{ij}\in R\text{ for }1\leq i,j\leq n).
\end{equation}
Let $S_n$ denote the group of permutations of the index set $\{1,\dots,n\}$, and let $I$ denote an arbitrary subset of $\{1,\dots,n\}$.  We have
\begin{equation}\label{split_off_n_minus_2}
\begin{split}
	\det (A) & = \sum_{\sigma\in S_n} {\sgn(\sigma) \prod_{i=1}^n (1 + \pi^\ell a_{i\sigma(i)})} \\
	& = \sum_{\sigma\in S_n} {\sgn(\sigma) \sum_{\, I\subseteq \{1,\dots,n\} \, } {\prod_{i\in I} {\pi^\ell a_{i\sigma(i)}}}} \\
	& = \sum_{k=0}^n {\sum_{\substack{ \, I\subseteq\{1,\dots,n\}  \\   |I| = k }} {\pi^{\ell k} S(I) }},
\end{split}
\end{equation}
where for each $I\subseteq\{1,\dots,n\}$ we define
\begin{align*}
S(I) = \sum_{\sigma\in S_n} { \sgn(\sigma) \prod_{i\in I} {a_{i\sigma(i)}} }.
\end{align*}

We will now show that 
\begin{equation}\label{SVanishes}
S(I)  = 0 \hskip1cm (0\leq |I|\leq n-2).
\end{equation}
For if $0\leq |I|\leq n-2$, then the set $\{1,\dots,n\}\setminus I$ has at least two elements, so there exists a transposition $\epsilon\in S_n$ that fixes each element of $I$. Then since the map $\sigma\mapsto \sigma\epsilon$ is a bijection from $S_n$ to itself, we have
\begin{equation*}
\begin{split}
S(I) & = \sum_{\sigma\in S_n} { \sgn(\sigma\epsilon) \prod_{i\in I} {a_{i{\sigma\epsilon}(i)}} } \\
& = \sgn(\epsilon) \sum_{\sigma\in S_n} { \sgn(\sigma) \prod_{i\in I} {a_{i{\sigma}(i)}} } \label{sigma_epsilon} \\
& = -S(I).
\end{split}
\end{equation*}
which implies $(\ref{SVanishes})$. (Of course, the equality $x=-x$ does not imply $x=0$ in a field of characteristic $2$.  However, one may view the proof of $(\ref{SVanishes})$ as taking place in the polynomial ring $\ZZ[a_{ij}]$ in $n^2$ doubly indexed indeterminates $a_{ij}$, which implies that the identity holds in arbitrary characteristic.)

Because of $(\ref{SVanishes})$ the identity \eqref{split_off_n_minus_2} 
becomes
\begin{equation}\label{simpified_S_of_I}
\det(A) = \sum_{k=n-1}^{n} {\sum_{\substack{ \, I\subseteq\{1,\dots,n\} \\   |I| = k }} {\pi^{\ell k} S(I)}}.
\end{equation}
When $n-1\leq k \leq n$, it follows from the strong triangle inequality that 
\begin{equation}\label{ILargeBound}
|\pi^{\ell k}S(I)| \leq |\pi|^{\ell k}\leq |\pi|^{\ell(n-1)},
\end{equation}
and we conclude from $(\ref{simpified_S_of_I})$ and $(\ref{ILargeBound})$ that $|\det A|\leq|\pi|^{\ell(n-1)}$.
\end{proof}

\begin{lem}\label{MainCoeffLem}
Let $A$ be an $n\times n$ matrix over $K$ with all entries in the disc $D(1,|\pi|^\ell)$ for some positive integer $\ell$, and let
\begin{equation*}
f(x) = x^n+c_{n-1}x^{n-1}+\dots+c_1x+c_0
\end{equation*}
be the characteristic polynomial of $A$.  Then 
\begin{equation}\label{SmallCoefficients}
|c_j|\leq|\pi|^{\ell(n-j-1)} \hskip1cm (0\leq j\leq n-1).
\end{equation}
Moreover, if $\ell>\ord_\pi(n)$, then 
\begin{equation}\label{SmallCoefficientsTraceId}
|c_{n-1}|=|n|.
\end{equation}
\end{lem}

\begin{proof}
First we establish some notation.  Given positive integers $i$ and $j$, let $\delta_{ij}$ represent the Kronecker delta function of $i$ and $j$. Given a permutation $\sigma\in S_n$, let $\Fix(\sigma)=\{i\in\{1,\dots,n\}\mid \sigma(i)=i\}$ denote the set of indices that are fixed by $\sigma$. Given a subset $I\subseteq\{1,\dots,n\}$, let $I^c=\{1,\dots,n\}\setminus I$.

Again with $A$ written as in $(\ref{AWrittenNear1})$, we have
\begin{equation}\label{CharPolyCalc1}
\begin{split}
f(x) & = \det(\delta_{ij}x-(1+\pi^\ell a_{ij})) \\
	& = \sum_{\sigma\in S_n}\sgn(\sigma)\prod_{i=1}^n(\delta_{i\sigma(i)}x-(1+\pi^\ell a_{i\sigma(i)})) \\
	& =\sum_{\sigma\in S_n}\sgn(\sigma)\sum_{I\subseteq \{1,\dots,n\}}\bigg(\prod_{i\in I}\delta_{i\sigma(i)}x\bigg)\bigg(\prod_{i\in I^c}(-1)(1+\pi^\ell a_{i\sigma(i)})\bigg) \\ 
	& = \sum_{\sigma\in S_n}\sgn(\sigma)\sum_{I\subseteq \Fix(\sigma)}(-1)^{|I^c|}\bigg(\prod_{i\in I^c}(1+\pi^\ell a_{i\sigma(i)})\bigg)x^{|I|} \\
	& = \sum_{I\subseteq \{1,\dots,n\}}T(I)x^{|I|}\end{split}
\end{equation}
where
\begin{equation}\label{CharPolyCalcCoeff}
T(I)= (-1)^{|I^c|}\sum_{\stackrel{\sigma\in S_n}{I\subseteq \Fix(\sigma)}}\sgn(\sigma)\prod_{i\in I^c}(1+\pi^\ell a_{i\sigma(i)}).
\end{equation}
The second-to-last equality in $(\ref{CharPolyCalc1})$ follows from the fact that, because of the Kronecker delta factors, the $(\sigma,I)$-th summand vanishes unless $\sigma$ fixes every element of $I$.  

Note that if we let $A_I$ denote the $(n-|I|)\times(n-|I|)$ matrix formed by removing the $i$-th row and $j$-th column from $A$ for each $i,j\in I$, then $(\ref{CharPolyCalcCoeff})$ becomes
\begin{equation}\label{CharPolyCalcCoeff2}
T(I)= (-1)^{|I^c|}\sum_{\sigma\in \Perm(I^c)}\sgn(\sigma)\prod_{i\in I^c}(1+\pi^\ell a_{i\sigma(i)}) = (-1)^{|I^c|}\det A_I.
\end{equation}
It follows that for $0\leq j\leq n-1$, the coefficient $c_j$ of $f(x)$ can be calculated by adding and subtracting the determinants of $(n-j)\times(n-j)$ matrices with entries in $1 + \pi^\ell R$; Lemma~\ref{MainDetLem} implies that each of these determinants has an absolute value of at most $|\pi|^{\ell(n-j-1)}$, and the desired bound $(\ref{SmallCoefficients})$ follows.

Finally, if $\ell>\ord_\pi(n)$, then since $c_{n-1}$ is the negative of the trace of $A$, we have
\begin{equation*}
|c_{n-1}| = |n+\pi^\ell(a_{11}+\dots+a_{nn})|=|n|
\end{equation*}
by $(\ref{STIEquality})$ and the fact that $|\pi^\ell(a_{11}+\dots+a_{nn})|\leq |\pi|^\ell<|n|$.
\end{proof}

\section{The main result}\label{MainResultSect}

\begin{thm}\label{MainThmGeneral}
Let $A$ be an $n\times n$ matrix over $K$ with all entries in the disc $D(1,|\pi|^\ell)$ for a positive integer $\ell>2\,\ord_\pi(n)$.  Then the following conditions hold.
\begin{itemize}
	\item[{\bf (a)}]  $A$ has a simple eigenvalue $\lambda_\max\in K$ such that $|\lambda|<|\lambda_\max|$ for all other eigenvalues $\lambda\in\CC_K$ of $A$.  Moreover, $|\lambda_\max|=|n|$.
	\item[{\bf (b)}]  The maximal eigenvalue $\lambda_\max$ is an element of the disc $D(n,|\pi|^\ell/|n|)$. 
	\item[{\bf (c)}]  There exists a $\lambda_\max$-eigenvector $\x$ of $A$, all of whose components are elements of the disc $D(1,|\pi|^\ell/|n|)$.
	\item[{\bf (d)}]  The limit
\begin{equation*}
P_A=\lim_{k\to+\infty}\textstyle(\frac{1}{\lambda_\max}A)^k
\end{equation*} 
exists and $P_A$ is a projection operator onto the $\lambda_\max$-eigenspace of $A$.
\end{itemize}
\end{thm}

\begin{rem}
Theorem~\ref{MainThmIntro} is the special case of Theorem~\ref{MainThmGeneral} in which $K=\QQ_p$, $\pi=p$, $\ell=1$, and $p\nmid n$.
\end{rem}

\begin{rem}
It is implicit in the hypothesis $\ell>2\,\ord_\pi(n)$ that $n$ is not divisible by the characteristic of $K$, for if $\ch(K)\mid n$ then one has $\ord_\pi(n)=\ord_\pi(0)=+\infty$.
\end{rem}

\begin{rem}
The hypothesis $\ell>2\,\ord_\pi(n)$ can be written in the form $|n|>|\pi|^{\ell/2}$, which implies in particular that 
\begin{equation}\label{BoundsAreNontrivial}
|\pi|^{\ell}/|n|<|\pi|^{\ell/2}<1.
\end{equation}
The inequalities $(\ref{BoundsAreNontrivial})$ show that the bounds quoted in parts (b) and (c) of the Theorem are nontrivial and can be made independent of the dimension $n$.
\end{rem}

\begin{proof}[Proof of Theorem~\ref{MainThmGeneral} (a)]
In order to show the existence of a strictly maximal eigenvalue $\lambda_\max\in K$ of $A$ of multiplicity one, we require an analysis of the characteristic polynomial 
\begin{equation*}
f(X)=x^n+c_{n-1}x^{n-1}+\dots+c_1x+c_0
\end{equation*}
of $A$ via its Newton polygon.  Recall that the Newton polygon of $f(x)$ is the lower convex hull of the set of points 
\begin{equation*}
(0,\ord_\pi(c_0)), (1,\ord_\pi(c_1)), \dots, (n-1,\ord_\pi(c_{n-1})), (n,0)
\end{equation*}
in the $xy$-plane.  (See \cite{MR861410} $\S$6.3 or \cite{MR1488696} $\S$6.4.)  A partial Newton polygon for $f(X)$ is depicted in Figure~\ref{newton2}.

\begin{figure}[hbt]
\begin{center}
\includegraphics{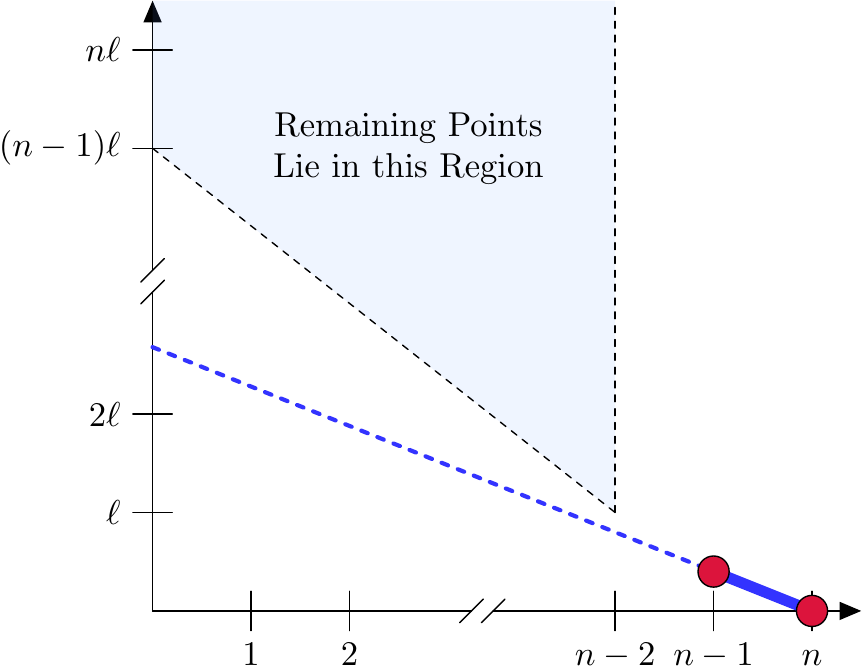}
\caption{The Newton polygon for $f(X)$.}
\label{newton2}
\end{center}
\end{figure}

From Lemma~\ref{MainCoeffLem} and the hypothesis $\ell>2\,\ord_\pi(n)$, we have 
\begin{equation}\label{SmallCoeffs}
\ord_\pi(c_{i}) \geq \ell(n-i-1) \hskip1cm (0 \leq i \leq n-2)
\end{equation}
and
\begin{equation}\label{LargeCoeff}
\ord_\pi(c_{n-1}) = \ord_\pi(n)<\ell/2.
\end{equation}
Observe that the line segment connecting $(n-1, \ord_\pi(c_{n-1}))$ with $(n,0)$ has slope
\begin{equation*}
m=\frac{0 - \ord_\pi(c_{n-1})}{n-(n-1)} = -\ord_\pi(c_{n-1}) =-\ord_\pi(n)> -\ell/2.
\end{equation*}
Extending this line segment to the left until it intersects the $y$ axis, this line, call it $L$, has equation $y=\ord_\pi(n)(n-x)$.  For $0 \leq i \leq n-2$, we bound the $y$-coordinate of the point on the line $L$ with $x$-coordinate equal to $i$:
\begin{equation}\label{NewtonPolySlopeCalc}
\begin{split}
\ord_\pi(n) (n-i) &< (\ell/2) (n-i) \\
& =  (\ell/2) (n-i-1) + \ell/2 \\
&\leq \ord_\pi(c_{i})/2 + \ell/2 \\
&\leq \ord_\pi(c_{i})/2 + \ord_\pi(c_{i})/2 \\
&= \ord_\pi(c_{i}).
\end{split}
\end{equation}
Here we have used $(\ref{LargeCoeff})$, then $(\ref{SmallCoeffs})$, and finally the inequality $\ell\leq\ord_\pi(c_i)$, which follows from $(\ref{SmallCoeffs})$ and the fact that $n-i-1$ is a positive integer.

The inequality $(\ref{NewtonPolySlopeCalc})$ shows that the line $L$ lies strictly below all of the left-most $n-1$ points $(i,\ord_\pi(c_i))$, $0\leq i\leq n-2$, of the Newton polygon of $f(x)$ (see Figure~\ref{newton2}).  Since the distinct slopes of the segments of the Newton polygon are increasing from left to right, it follows that the slopes of all of the segments of the Newton polygon $f(x)$ to the left of the point $(n-1, \ord_\pi(c_{n-1}))$ are strictly less than the slope $m=-\ord_\pi(c_{n-1})$ of the segment from $(n-1, \ord_\pi(c_{n-1}))$ to $(n,0)$.  Since this final segment has horizontal length $1$, it follows from the Theorem of the Newton Polygon (\cite{MR861410} Thm. 6.3.1)  that $f(x)$ factors over $K$ as $f(x)=g(x)(x-\lambda_\max)$ for $\lambda_\max\in K$ satisfying 
\begin{equation}\label{MaxEigenIdent}
|\lambda_\max|=|\pi|^{-m}=|\pi|^{\ord_\pi(c_{n-1})}=|c_{n-1}|=|n|,
\end{equation}
and that all roots $\lambda$ of $g(x)$ in $\CC_K$ have absolute value $|\lambda|<|\lambda_\max|$.
\end{proof}

\begin{proof}[Proof of Theorem~\ref{MainThmGeneral} (b)]
Let $\lambda_1,\dots,\lambda_{n-1}\in \CC_K$ be the non-maximal eigenvalues of $A$.  Then $|n-\lambda_i|=|n|$ for each $1\leq i\leq n-1$, by $(\ref{STIEquality})$ and the fact that $|\lambda_i|<|\lambda_\max|=|n|$.  So on the one hand we have
\begin{equation*}
\begin{split}
|f(n)|= |n-\lambda_1|\dots|n-\lambda_{n-1}||n-\lambda_\max|=|n|^{n-1}|n-\lambda_\max|.
\end{split}
\end{equation*}
On the other hand, with $A$ written as in $(\ref{AWrittenNear1})$, using Lemma~\ref{MainCoeffLem} and the fact that $c_{n-1}$ is the negative of the trace of $A$, we have
\begin{equation*}
\begin{split}
|f(n)| & = |n^n -(n+\pi^\ell(a_{11}+\dots+a_{nn}))n^{n-1} + c_{n-2}n^{n-2}+\dots+c_1n+c_0| \\
	& = |-\pi^\ell(a_{11}+\dots+a_{nn})n^{n-1} + c_{n-2}n^{n-2}+\dots+c_1n+c_0| \\
	& \leq\max(|\pi|^\ell|n|^{n-1},|c_{n-2}||n|^{n-2},\dots,|c_1||n|,|c_0|) \\
	& \leq\max(|\pi|^{\ell}|n|^{n-1},|\pi|^{\ell}|n|^{n-2},|\pi|^{2\ell}|n|^{n-3},\dots,|\pi|^{(n-2)\ell}|n|,|\pi|^{(n-1)\ell}) \\
	& =|\pi|^{\ell}|n|^{n-2}.
\end{split}
\end{equation*}
In this inequality, the evaluation of the maximum as $|\pi|^{\ell}|n|^{n-2}$ follows from an elementary argument and the hypothesis that $\ell>2\,\ord_\pi(n)$, which implies $|n|>|\pi|^{\ell/2}$.  Combining these last two calculations for $|f(n)|$ we deduce the desired inequality 
$|n-\lambda_\max|\leq |\pi|^{\ell}/|n|$.  
\end{proof}

\begin{proof}[Proof of Theorem~\ref{MainThmGeneral} (c)]
Let $\v=(v_i)\in K^n$ be an arbitrary $\lambda_\max$-eigenvector.  Normalizing by a suitable nonzero scalar, we may assume without loss of generality that $\|\v\|=1$.  Let $\1$ denote the $n$-dimensional column vector, all of whose entries are equal to $1$, and let $1_{n\times n}$ denote the $n\times n$ matrix, all of whose entries are equal to $1$.  With $A$ written as in $(\ref{AWrittenNear1})$, we have
\begin{equation*}
A=(1+\pi^\ell a_{ij})=1_{n\times n}+\pi^\ell A'
\end{equation*}
where $A'=(a_{ij})$ and $\|A'\|\leq1$.  Thus
\begin{equation*}
\begin{split}
\lambda_\max\v & = A\v \\
	& = 1_{n\times n}\v + \pi^\ell A'\v \\
	& = v\1 + \pi^\ell A'\v,
\end{split}
\end{equation*}
where $v=v_1+v_2+\dots+v_n$.  We then have
\begin{equation*}
\begin{split}
|v/\lambda_\max| & = \|(v/\lambda_\max)\1\| \\
	& = \|\v-(\pi^\ell/\lambda_\max)A'\v\| \\
	& = 1.
\end{split}
\end{equation*}
This last equality follows from $(\ref{STIEquality})$, because the largest entry of $\v$ has absolute value $1$, while all entries of $(\pi^\ell/\lambda_\max)A'\v$ have absolute value at most
\begin{equation*}
\|(\pi^\ell/\lambda_\max)A'\v\| \leq |\pi^\ell/\lambda_\max|\|A'\|\|\v\|\leq |\pi|^\ell/|n|\leq |\pi|^{\ell/2}<1
\end{equation*}
using $(\ref{BoundsAreNontrivial})$.  In particular, $v$ is nonzero and $|v|=|\lambda_\max|=|n|$.

We now set $\x=(\lambda_\max/v)\v$.  Then $\x$ is a $\lambda_\max$-eigenvector of $A$ and 
\begin{equation*}
\|\x-\1\|=\|(\pi^\ell/v) A'\v\|\leq (|\pi|^\ell/|n|)\|A'\|\|\v\|\leq|\pi|^\ell/|n|
\end{equation*}
as desired.\end{proof}

In view of part (a) of Theorem~\ref{MainThmGeneral}, part (d) follows from the following theorem which holds in somewhat greater generality.

\begin{thm}
Let $A$ be an $n\times n$ matrix with entries in $K$, and suppose that $A$ has a simple eigenvalue $\lambda_\max\in K$ with the property that $|\lambda|<|\lambda_\max|$ for all other eigenvalues $\lambda\in\CC_K$ of $A$.  Then the limit
\begin{equation}\label{MatrixLimit}
P_A=\lim_{k\to+\infty}\textstyle(\lambda_\max^{-1}A)^k
\end{equation} 
exists and $P_A\in M_{n\times n}(K)$ is a projection operator onto the $\lambda_\max$-eigenspace of $A$.
\end{thm}
\begin{proof}
It suffices to prove that the limit $(\ref{MatrixLimit})$ exists in the space $M_{n\times n}(\CC_K)$ endowed with the norm $\|\cdot\|$ defined in $(\ref{NormDef})$.  Plainly this is equivalent to componentwise convergence, and the fact that such a limit $P_A$ must have entries in $K$ follows from the completeness of $K$.

We have $A=Q^{-1}JQ$, where $Q\in M_{n\times n}(\CC_K)$ is a nonsingular change of basis matrix, and
\begin{align*}
J &=
\begin{pmatrix}
\lambda_\max  & 0      & \hdots & 0      \\
0         & J_1    & \hdots & 0      \\
\vdots    & \vdots & \ddots & \vdots \\ 
0         & 0      & \hdots & J_r
\end{pmatrix}
\end{align*}
is a Jordan canonical form for $A$ (written here in block form); here $\lambda_\max$ is the maximal eigenvalue guaranteed by Theorem~\ref{MainThmGeneral} part (a), and each $J_i$ is an $m_{i} \times m_{i}$ Jordan block
\begin{align*}
J_{i} &=
\begin{pmatrix}
\lambda_i & 1         & \hdots & 0         & 0         \\
0         & \lambda_i & \hdots & 0         & 0         \\
\vdots    & \vdots    & \ddots & \vdots    & \vdots    \\ 
0         & 0         & \hdots & \lambda_i & 1         \\
0         & 0         & \hdots & 0         & \lambda_i \\
\end{pmatrix}.
\end{align*}
Thus the $\lambda_1,\dots,\lambda_r$ denote the (not necessarily distinct) nonmaximal eigenvalues of $A$.  Given a positive integer $k$, we have 
\begin{align*}
J^k &=
\begin{pmatrix}
\lambda_\max ^{k} & 0      & \hdots & 0      \\
0         & J_{1}^k    & \hdots & 0      \\
\vdots    & \vdots & \ddots & \vdots \\ 
0         & 0      & \hdots & J_{r}^k
\end{pmatrix}
\end{align*}
and a straightforward induction argument gives the well-known formula
\begin{align*}
J_{i}^{k} &=
\renewcommand*{\arraystretch}{1.3}
\arraycolsep=7pt
\begin{pmatrix}
\lambda_i^{k} & {k \choose 1} \lambda_i^{k-1} & \hdots & {k \choose m_{i}-2} \lambda_i^{k-(m_{i}-2)} & {k \choose m_{i}-1} \lambda_i^{k-(m_{i}-1)} \\
0         & \lambda_i^{k} & \hdots & {k \choose m_{i}-3} \lambda_i^{k-(m_{i}-3)} & {k \choose m_{i}-2} \lambda_i^{k-(m_{i}-2)} \\
\vdots    & \vdots    & \ddots & \vdots    & \vdots    \\ 
0         & 0         & \hdots & \lambda_i^{k} & {k \choose 1} \lambda_i^{k-1} \\
0         & 0         & \hdots & 0         & \lambda_i^{k} \\
\end{pmatrix}
\end{align*}
for the $k$-th power of an $m_{i} \times m_{i}$ Jordan block; here we interpret ${k \choose m} = 0$ whenever $k<m$, as is standard.  We now divide through by $\lambda_\max ^k$ to obtain
\begin{align*}
(\lambda_\max^{-1}J)^k &=
\begin{pmatrix}
1 & 0      & \hdots & 0      \\
0         & (\lambda_\max^{-1}J_1)^k    & \hdots & 0      \\
\vdots    & \vdots & \ddots & \vdots \\ 
0         & 0      & \hdots & (\lambda_\max^{-1}J_r)^k 
\end{pmatrix}
\end{align*}
and
\begin{align*}
(\lambda_\max^{-1}J_i)^k &=
\renewcommand*{\arraystretch}{1.3}
\arraycolsep=7pt
\begin{pmatrix}
\mu_i^{k} & {k \choose 1} \mu_i^{k-1} \lambda_\max^{-1} & \hdots & {k \choose m_{i}-1} \mu_i^{k-(m_{i}-1)} \lambda_\max^{-(m_{i}-1)} \\
0         & \mu_i^{k} & \hdots & {k \choose m_{i}-2} \mu_i^{k-(m_{i}-2)} \lambda_\max^{-(m_{i}-2)} \\
\vdots    & \vdots    & \ddots & \vdots    \\ 
0         & 0         & \hdots & {k \choose 1} \mu_i^{k-1} \lambda_\max^{-1} \\
0         & 0         & \hdots & \mu_i^{k} \\
\end{pmatrix}
\end{align*}
where for each $1\leq i\leq r$ we define $\mu_i=\lambda_i/\lambda_\max$.

Noting that $|\mu_i|=|\lambda_{i}/\lambda_{\max}|<1$ for $1 \leq i \leq r$, and recalling that binomial coefficients are integers and therefore have non-Archimedean absolute value at most $1$, we see that each block $(\lambda_\max^{-1}J_i)^k$ converges to the zero matrix as $k\to+\infty$. We conclude that 
\begin{equation*}
\lim_{k\to+\infty}(\lambda_\max^{-1}J)^k=B
\end{equation*}
where
\begin{equation*}
B=\begin{pmatrix}
1 & 0      & \hdots & 0      \\
0         & 0    & \hdots & 0      \\
\vdots    & \vdots & \ddots & \vdots \\ 
0         & 0      & \hdots & 0 
\end{pmatrix}.
\end{equation*}
Now define 
\begin{equation}\label{ProjectionCalc}
P_A=Q^{-1}BQ.
\end{equation}
Using $(\ref{NormBound})$ we have 
\begin{equation*}
\begin{split}
\|(\lambda_\max^{-1}A)^k-P_A\| & = \|Q^{-1}((\lambda_\max^{-1}J)^k-B)Q\| \\
	& \leq \|Q^{-1}\|\|Q\|\|(\lambda_\max^{-1}J)^k-B)\|\to0 \\
\end{split}
\end{equation*}
as $k\to+\infty$, completing the proof of $(\ref{MatrixLimit})$.  We have $P_A^2=Q^{-1}B^2Q=Q^{-1}BQ=P_A$, and therefore $P_A$ is a projection operator.  Finally, given any vector $\x\in K^n$, we have
\begin{equation*}
\begin{split}
AP_A\x & = A\lim_{k\to+\infty}(\lambda_\max^{-1}A)^k\x \\
	& = \lambda_\max\lim_{k\to+\infty}(\lambda_\max^{-1}A)^{k+1}\x \\
	& = \lambda_\max P_A\x
\end{split}
\end{equation*}
and thus $P_A\x$ is a $\lambda_\max$-eigenvector; therefore the image of $P_A$ is the $\lambda_\max$-eigenspace of $K^n$.
\end{proof}

\section{A Counterexample when $\ell=2\,\ord_\pi(n)$}\label{SharpExampleSect}

In this section we give an example showing that Theorem~\ref{MainThmGeneral} is sharp, in the sense that the hypothesis $\ell>2\,\ord_\pi(n)$ cannot be relaxed.

Let $p$ be a prime, let $K=\QQ_p$, and let $n$ be a positive integer such that $p\mid n$.  Set $\ell=2\,\ord_p(n)$, and consider the $n\times n$ matrix 
\begin{equation*}
\begin{split}
A & =(1+p^\ell a_{ij}) \\
a_{ij} & = 
\begin{cases}
	1 & \text{ if $i=j=1$} \\
	0 & \text{ otherwise.}
\end{cases}
\end{split}
\end{equation*}
In other words, all entries of $A$ are equal to $1$ except the upper-left entry, which is equal to $1+p^\ell$.  In particular, all entries of $A$ are elements of the disc $D(1,p^{-\ell})$, as required in the hypotheses of Theorem~\ref{MainThmGeneral}.  Since $A$ has rank $2$, its characteristic polynomial $f(x)$ vanishes to multiplicity at least $n-2$.  We then have 
\begin{equation*}
f(x) = x^n+c_{n-1}x^{n-1}+c_{n-2}x^{n-2}
\end{equation*}
where
\begin{equation}\label{ExampleCoeff1}
c_{n-1} = -\tr(A)=-(n+p^\ell),
\end{equation}
and we will show that
\begin{equation}\label{ExampleCoeff2}
c_{n-2} = (n-1)p^\ell.
\end{equation}
In particular, the eigenvalues of $A$ in $\CC_p$ are $0,\lambda_1,\lambda_2$, where $\lambda_1,\lambda_2$ are the roots of the quadratic polynomial
\begin{equation*}
g(x) = x^2+c_{n-1}x+c_{n-2}.
\end{equation*}
In fact we have $\lambda_1\neq\lambda_2$.  This can be seen by an elementary argument, showing that the discriminant of $g$ is positive; alternatively, all entries of $A$ are positive, so the classical Perron-Frobenius theorem implies that $A$ must have a simple nonzero eigenvalue.

From $(\ref{ExampleCoeff1})$ and $(\ref{ExampleCoeff2})$ we have $\ord_p(c_{n-1})=\ord_p(n)=\ell/2$ and $\ord_p(c_{n-2})=\ord_p(p^\ell)=\ell$.  It follows from the Newton polygon of $g(x)$ that $|\lambda_1|_p=|\lambda_2|_p=p^{-\ell/2}$, and thus $A$ has no strictly maximal eigenvalue.  

Finally, we include the calculation $(\ref{ExampleCoeff2})$.  By $(\ref{CharPolyCalc1})$ and $(\ref{CharPolyCalcCoeff})$ we have
\begin{equation*}
c_{n-2} = \sum_{\stackrel{I\subseteq \{1,\dots,n\}}{|I|=n-2}}T(I),
\end{equation*}
where for each subset $I\subseteq \{1,\dots,n\}$ with $|I|=n-2$,
\begin{equation*}
T(I)= \prod_{i\in I^c}(1+p^\ell a_{ii})-\prod_{i\in I^c}(1+p^\ell a_{i\tau(i)})
\end{equation*}
where $\tau\in S_n$ is the transposition which swaps the two elements of $I^c$.  If $1\in I$, then because of our choice of $a_{ij}$ we have $T(I)=1-1=0$.  If $1\not\in I$, say $I^c=\{1,i_0\}$ for $i_0\neq 1$, then $T(I)=(1+p^\ell)-(1)=p^\ell$.  The number of subsets $I$ of $\{1,\dots,n\}$ of size $|I|=n-2$ satisfying $1\not\in I$ is $\binom{n-1}{n-2}=n-1$, and therefore
\begin{equation*}
c_{n-2} = \sum_{\stackrel{I\subseteq \{1,\dots,n\}}{|I|=n-2}}T(I)=(n-1)p^\ell,
\end{equation*}
as desired.



\end{document}